\documentclass[12pt,reqno]{amsart}
\usepackage{amsmath,amssymb,amsfonts,amscd,latexsym,amsthm,mathrsfs,verbatim,color}
\usepackage[unicode]{hyperref}
\usepackage{hypbmsec}
\newtheorem{lemma}{Lemma}
\newtheorem{theorem}{Theorem}
\newtheorem{corollary}{Corollary}
\newtheorem{proposition}{Proposition}
\theoremstyle{definition}
\newtheorem{Def}{Definition}
\theoremstyle{remark}

\newtheorem*{remarks}{\bf Remarks}
\newtheorem{construction}{\bf Construction}
\newtheorem*{connection}{\bf Connection between the constructions}

\let\ol\overline

\let\lf\lfloor
\let\rf\rfloor
\let\lc\lceil
\let\rc\rceil
\renewcommand{\d}{{\mathrm d}}

\newcommand{\Res}{\operatornamewithlimits{Res}}

\renewcommand{\Re}{\operatorname{Re}}
\newcommand\ba{{\boldsymbol a}}
\newcommand\bb{{\boldsymbol b}}

\newcommand {\N}{\mathbb{N}}
\newcommand {\Q}{\mathbb{Q}}
\newcommand {\R}{\mathbb{R}}
\newcommand {\Z}{\mathbb{Z}}

\newcommand {\ve}{\varepsilon}
\newcommand {\vphi}{\varphi}

\begin{document}

\title{On simultaneous diophantine approximations to~$\zeta(2)$~and~$\zeta(3)$}

\author{Simon Dauguet}
\address{Universit\'e Paris-Sud, Laboratoire de Math\'ematiques d'Orsay, Orsay Cedex, F-91405, France}
\email{simon.dauguet@math.u-psud.fr}

\author{Wadim Zudilin}
\address{School of Mathematical and Physical Sciences,
The University of Newcastle, Callaghan NSW 2308, AUSTRALIA}
\email{wzudilin@gmail.com}

\date{\today}

\thanks{The second author is supported by the Australian Research Council.}

\begin{abstract}
We present a hypergeometric construction of rational approximations to $\zeta(2)$ and $\zeta(3)$
which allows one to demonstrate simultaneously the irrationality of each of the zeta values, as well
as to estimate from below certain linear forms in 1, $\zeta(2)$ and $\zeta(3)$ with rational coefficients.
We then go further to formalise the arithmetic structure of these specific linear forms by
introducing a new notion of (simultaneous) diophantine exponent. Finally, we study the
properties of this newer concept and link it to the classical irrationality exponent and its
generalisations given recently by S.~Fischler.
\end{abstract}

\maketitle

\section{Introduction}

It is known that the Riemann zeta function $\zeta(s)$ takes irrational values at positive even integers.
This follows from Euler's evaluation $\zeta(s)/\pi^s\in\mathbb Q$ for $s=2,4,6,\dots$ and from
the transcendence of~$\pi$. Less is known about the values of $\zeta(s)$ at odd integers $s>1$.
Ap\'ery was the first to establish the irrationality of such a zeta value $\zeta(s)$:
he proved \cite{Apery} in~1978 that $\zeta(3)$ is irrational. The next major step in the direction
was made by Ball and Rivoal \cite{BallRivoal} in~2000: they showed that there are infinitely
many odd integers at which Riemann zeta function is irrational. Shortly after, Rivoal demonstrated \cite{Rivoal}
that one of the nine numbers $\zeta(5),\zeta(7),\dots,\zeta(21)$ is irrational, while
the second author \cite{Zud01} reduced the nine to four: he proved that at least one
of the four numbers $\zeta(5)$, $\zeta(7)$, $\zeta(9)$ and $\zeta(11)$ is irrational.

Already in 1978, Ap\'ery constructs linear forms in 1 and $\zeta(2)$, as well as in 1 and $\zeta(3)$,
with integer coefficients that produce the irrationality of the two zeta values in a quantitative form:
the constructions imply upper bounds $\mu(\zeta(2))< 11.850878\dots$ and $\mu(\zeta(3))<13.41782\dots$
for the irrationality measures. Recall that the irrationality exponent $\mu(\alpha)$ of a real irrational
$\alpha$ is the supremum of the set of exponents $\mu$ for which the inequality $|\alpha-p/q|<q^{-\mu}$
has infinitely many solutions in rationals $p/q$. Hata improves the above mentioned results
to $\mu(\zeta(2))<5.687$ in \cite[Addendum]{HataZeta2} and to $\mu(\zeta(3))<7.377956\dots$ in~\cite{HataZeta3}.
Further, Rhin and Viola study a permutation group related to $\zeta(2)$ in
\cite{RhinViolaZeta2} and show that $\mu(\zeta(2))<5.441243$. They later apply their new permutation group
arithmetic method to $\zeta(3)$ as well, to prove the upper bound $\mu(\zeta(3))<5.513891$.
In an attempt to unify the achievements of Ball--Rivoal and of Rhin--Viola, the second author re-interpreted
the constructions using the classical theory of hypergeometric functions and integrals \cite{Zu04}.
In his recent work \cite{Zu13}, he uses the permutation group arithmetic method and
a hypergeometric construction, closely related to the one in this paper, to sharpen the earlier
irrationality exponent of $\zeta(2)$ to $\mu(\zeta(2))\le5.09541178\dots$\,.

\medskip
In this paper, we construct simultaneous rational approximations to both $\zeta(2)$ and $\zeta(3)$ using hypergeometric tools,
and establish from them a lower bound for $\Q$-linear combinations of 1, $\zeta(2)$ and $\zeta(3)$ under some strong divisibility
conditions on the coefficients. Namely, we prove

\begin{theorem}\label{th}
Let $\eta$ and $\ve$ be positive real numbers. For $m$ sufficiently large with respect to $\ve$ and $\eta$,
let $(a_0,a_1,a_2)\in\Q^3\setminus\{\bold0\}$ be such that
\begin{enumerate}
\item[(i)] $D_{m}^2D_{2m}a_0\in\Z$, $D_{m}a_1\in\Z$ and $\dfrac{D_{2m}}{D_{m}}\,a_2\in\Z$,
where $D_m$ denotes the least common multiple of $1,2,\dots,m$; and
\item[(ii)] $|a_0|,|a_1|,|a_2|\leq e^{-(\tau_0+\ve)m}$ hold with $\tau_0=0.899668635\dots$\,.
\end{enumerate}
Then $|a_0+a_1\zeta(2)+a_2\zeta(3)|>e^{-(s_0+\eta)m}$ with $s_0=6.770732145\dots$\,.
\end{theorem}

Theorem \ref{th} contains the irrationality of both $\zeta(2)$ and $\zeta(3)$, because $\tau_0<1$.
Namely, taking
$$
a_0=\frac{-p}{D_{m}}, \quad a_1=\frac{q}{D_{m}} \quad\text{and}\quad a_2=0
$$
shows that $\zeta(2)\ne p/q$, while the choice
$$
a_0=\frac{-D_{m}p}{D_{2m}}, \quad a_1=0 \quad\text{and}\quad a_2=\frac{D_{m}q}{D_{2m}}
$$
implies that $\zeta(3)\ne p/q$. The theorem does not give however the expected linear independence of
$1$, $\zeta(2)$ and $\zeta(3)$: it remains an open problem.

\medskip
Our proof of Theorem~\ref{th} heavily rests upon a general version of hypergeometric construction
of linear forms in 1 and $\zeta(2)$ on one hand, and in 1 and $\zeta(3)$ on the other hand;
some particular instances of this construction were previously outlined in \cite{Zud11}.
More precisely, the linear forms $r_n=q_n\zeta(2)-p_n$ and $\hat{r}_n=\hat{q}_n\zeta(3)-\hat{p}_n$
we construct in the proof are hypergeometric-type series that depend on certain
sets of auxiliary integer parameters. Permuting parameters in the sets allows us
to gain $p$-adic information about the coefficients $q_n$,
$\hat{q}_n$, $p_n$ and $\hat{p}_n$. In addition, a classical transformation from the theory of
hypergeometric functions implies that $q_n=\hat{q}_n$. The latter fact leads us to \emph{simultaneous}
rational approximations $r_n=q_n\zeta(2)-p_n$ and $\hat{r}_n=q_n\zeta(3)-\hat{p}_n$
to $\zeta(2)$ and $\zeta(3)$, with the following arithmetical and asymptotic properties:
\begin{gather*}
\hat{\Phi}_n^{-1}q_n,\ \hat{\Phi}_n^{-1}D_{8n}D_{16n}p_n,\ \hat{\Phi}_n^{-1}D_{8n}^3\hat{p}_n\in\Z,
\\
\lim_{n\to\infty} \frac{\log\hat{\Phi}_n}{n}=\vphi=5.70169601\hdots,
\end{gather*}
where $\hat{\Phi}_n$ is an explicit product over primes, and
\begin{gather*}
\limsup_{n\to\infty}\frac{\log|r_n|}{n}=\limsup_{n\to\infty}\frac{\log|\hat{r}_n|}{n}=-\rho=-19.10095491\hdots,
\\
\lim_{n\to\infty}\frac{\log|q_n|}{n}=\kappa=27.86755317\hdots.
\end{gather*}
Finally, executing the Gosper--Zeilberger algorithm of creative telescoping
we find out a recurrence relation satisfied by the linear forms $r_n$ and $\hat{r}_n$.
Together with a standard argument using the nonvanishing determinants formed from the
coefficients of the forms, we then deduce Theorem~\ref{th} (some further computational details can be found in~\cite{these}).
Note that
\begin{equation}
\tau_0=\frac{1}{8}(32-\vphi-\rho)=0.899668635\dots
\quad\text{and}\quad
s_0=\frac{1}{8}(32-\vphi+\kappa)=6.770732145\dots,
\label{eq-C}
\end{equation}
and the integer $m$ from Theorem~\ref{th} is essentially $8n$.

In order to accommodate the atypical simultaneous approximations in Theorem~\ref{th} as well as to
relate them to the context of previous results listed in the beginning of the section,
we define a new diophantine exponent $s_\tau(\xi_1,\xi_2)$ of two real numbers $\xi_1$ and $\xi_2$,
a characteristic of simultaneous irrationality of the numbers which depends on an additional parameter~$\tau$.
With this notion in mind, we restate Theorem~\ref{th} as $s_{\tau_0}(\zeta(2),\zeta(3))\leq s_0$.
Exploiting further the properties of the exponent, we demonstrate in Proposition~\ref{majoration s_tau si lie}
the unlikeness of linear dependence of 1, $\zeta(2)$ and $\zeta(3)$ over~$\mathbb Q$:
the latter would imply $s_0=6-\tau_0$ or the belonging of both $\zeta(2)$ and $\zeta(3)$ to
a certain set of Lebesgue measure~0.

\medskip
In \S\,\ref{HS} we introduce hypergeometric tools which depend on some parameters that lead to
$\mathbb Q$-linear forms in 1 and $\zeta(2)$ on one hand,
and in 1 and $\zeta(3)$ one the other, the forms having some common asymptotic properties.

In \S\,\ref{z23} we specialise the parameters of the previous part to have the coefficients of $\zeta(2)$ and $\zeta(3)$ coincide.
{}From this specialisation we derive the main theorem using recurrence relations satisfied by the linear forms and their coefficients.

In the final part, \S\,\ref{NDE}, we introduce a new diophantine exponent. Some basic properties of this exponent are given,
and it is compared to the irrationality exponents previously known. Then the main result is restated in terms of
this diophantine exponent as Theorem~\ref{majoration s}, for consistency with previous results in the subject.

\section{Hypergeometric series}
\label{HS}

In what follows, we always assume standard hypergeometric notation~\cite{Sl}.
For $n\in\N$, the Pochhammer symbol is given by
$$
(a)_n=\frac{\Gamma(a+n)}{\Gamma(a)}=\prod_{k=0}^{n-1}(a+k),
$$
with the convention $(a)_0=1$, while the generalized hypergeometric function is defined by the series
$$
{}_{p+1}F_p\biggl(\begin{matrix} a_0, \, a_1, \, \dots, \, a_p \\ b_1, \, \dots, \, b_p\end{matrix}\biggm|z\biggr)
=\sum_{n=0}^\infty\frac{(a_0)_n(a_1)_n\dotsb(a_p)_n}{n!\,(b_1)_n\dotsb(b_p)_n}\,z^n.
$$

\subsection{Integer-valued polynomials}
\label{BSI}

We reproduce here some auxiliary results about integer-valued polynomials;
the proofs can be found in~\cite{Zu13}.

\begin{lemma}
\label{barnes}
For $\ell=0,1,2,\dots$,
\begin{equation}
\frac1{2\pi i}\int_{1/2-i\infty}^{1/2+i\infty}\biggl(\frac\pi{\sin\pi t}\biggr)^2
\frac{(t-1)(t-2)\dotsb(t-\ell)}{\ell!}\,\d t
=\frac{(-1)^\ell}{\ell+1}.
\label{expp}
\end{equation}
\end{lemma}

\begin{lemma}
\label{iv1}
Given $b<a$ integers, set
$$
R(t)=R(a,b;t)=\frac{(t+b)(t+b+1)\dotsb(t+a-1)}{(a-b)!}.
$$
Then
$$
R(k)\in\mathbb Z, \quad D_{a-b}\cdot\frac{\d R(t)}{\d t}\bigg|_{t=k}\in\mathbb Z
\quad\text{and}\quad
D_{a-b}\cdot\frac{R(k)-R(\ell)}{k-\ell}\in\mathbb Z
$$
for any $k,\ell\in\mathbb Z$, $\ell\ne k$.
\end{lemma}

\begin{lemma}
\label{iv2}
Let $R(t)$ be a product of several integer-valued polynomials
$$
R_j(t)=R(a_j,b_j;t)=\frac{(t+b_j)(t+b_j+1)\dotsb(t+a_j-1)}{(a_j-b_j)!}, \quad\text{where}\; b_j<a_j,
$$
and $m=\max_j\{a_j-b_j\}$. Then
\begin{equation}
R(k)\in\mathbb Z, \quad D_m\cdot\frac{\d R(t)}{\d t}\bigg|_{t=k}\in\mathbb Z
\quad\text{and}\quad
D_m\cdot\frac{R(k)-R(\ell)}{k-\ell}\in\mathbb Z
\label{iv-incl}
\end{equation}
for any $k,\ell\in\mathbb Z$, $\ell\ne k$.
\end{lemma}

\subsection({Construction of linear forms in 1 and \003\266(2)})%
{Construction of linear forms in $1$ and $\zeta(2)$}
\label{gc1}

The construction in this subsection is a general case of the one considered in~\cite[Section 2]{Zu07}.

For a set of parameters
\begin{equation*}
(\ba,\bb)=\biggl(\begin{matrix} a_1, \, a_2, \, a_3, \, a_4 \\ b_1, \, b_2, \, b_3, \, b_4 \end{matrix}\biggr)
\end{equation*}
subject to the conditions
\begin{equation}
\begin{gathered}
b_1,b_2,b_3\le a_1,a_2,a_3,a_4<b_4,
\\
d=(a_1+a_2+a_3+a_4)-(b_1+b_2+b_3+b_4)\ge 0,
\end{gathered}
\label{cond1}
\end{equation}
define the rational function
\begin{align}
R(t)
&=R(\ba,\bb;t)
=\frac{(t+b_1)\dotsb(t+a_1-1)}{(a_1-b_1)!}
\cdot\frac{(t+b_2)\dotsb(t+a_2-1)}{(a_2-b_2)!}
\nonumber\\ &\phantom{=R(\ba,\bb;t)} \qquad\times
\frac{(t+b_3)\dotsb(t+a_3-1)}{(a_3-b_3)!}
\cdot\frac{(b_4-a_4-1)!}{(t+a_4)\dotsb(t+b_4-1)}
\label{eq:gc}
\\
&=\Pi(\ba,\bb)\cdot\frac{\Gamma(t+a_1)\,\Gamma(t+a_2)\,\Gamma(t+a_3)\,\Gamma(t+a_4)}
{\Gamma(t+b_1)\,\Gamma(t+b_2)\,\Gamma(t+b_3)\,\Gamma(t+b_4)},
\label{eq:P0}
\end{align}
where
$$
\Pi(\ba,\bb)
=\frac{(b_4-a_4-1)!}{(a_1-b_1)!\,(a_2-b_2)!\,(a_3-b_3)!}.
$$
We also introduce the ordered versions $a_1^*\le a_2^*\le a_3^*\le a_4^*$ of the parameters $a_1,a_2,a_3,a_4$
and $b_1^*\le b_2^*\le b_3^*$ of $b_1,b_2,b_3$, so that $\{a_1^*,a_2^*,a_3^*,a_4^*\}$ coincides with $\{a_1,a_2,a_3,a_4\}$
and $\{b_1^*,b_2^*,b_3^*\}$ coincides with $\{b_1,b_2,b_3\}$ as multi-sets (that is, sets with possible repetition of elements).
Then $R(t)$ has poles at $t=-k$ where $k=a_4^*,a_4^*+1,\dots,b_4-1$,
zeroes at $t=-\ell$ where $\ell=b_1^*,b_1^*+1,\dots,a_3^*-1$, and double zeroes
at $t=-\ell$ where $\ell=b_2^*,b_2^*+1,\dots,a_2^*-1$.

Decomposing $R(t)$ into the sum of partial fractions, we get
\begin{equation}
R(t)=\sum_{k=a_4^*}^{b_4-1}\frac{C_k}{t+k}+P(t),
\label{eq:P1}
\end{equation}
where $P(t)$ is a polynomial of which the degree $d$ is defined in~\eqref{cond1} and
\begin{align}
C_k
&=\bigl(R(t)(t+k)\bigr)|_{t=-k}
\nonumber\\
&=(-1)^{d+b_4+k}\binom{k-b_1}{k-a_1}\binom{k-b_2}{k-a_2}\binom{k-b_3}{k-a_3}\binom{b_4-a_4-1}{k-a_4}\in\mathbb Z
\label{eq:P2}
\end{align}
for $k=a_4^*,a_4^*+1,\dots,b_4-1$.

\begin{lemma}
\label{lem:ap}
Set $c=\max\{a_1-b_1,a_2-b_2,a_3-b_3\}$. Then
$D_cP(t)$ is an integer-valued polynomial of degree $d$.
\end{lemma}

\begin{proof}
Write $R(t)=R_1(t)R_2(t)$, where
$$
R_1(t)=\frac{\prod_{j=b_1}^{a_1-1}(t+j)}{(a_1-b_1)!}
\cdot\frac{\prod_{j=b_2}^{a_2-1}(t+j)}{(a_2-b_2)!}
\cdot\frac{\prod_{j=b_3}^{a_3-1}(t+j)}{(a_3-b_3)!}
$$
is the product of three integer-valued polynomials and
$$
R_2(t)=\frac{(b_4-a_4-1)!}{\prod_{j=a_4}^{b_4-1}(t+j)}
=\sum_{k=a_4}^{b_4-1}\frac{(-1)^{k-a_4}\binom{b_4-a_4-1}{k-a_4}}{t+k}.
$$

It follows from Lemma~\ref{iv2} that
\begin{equation}
\begin{gathered}
D_c\cdot\frac{\d R_1(t)}{\d t}\bigg|_{t=j}\in\mathbb Z
\quad\text{for}\; j\in\mathbb Z
\quad\text{and}
\\
D_c\cdot\frac{R_1(j)-R_1(m)}{j-m}\in\mathbb Z
\quad\text{for}\; j,m\in\mathbb Z, \; j\ne m.
\end{gathered}
\label{eq:P3}
\end{equation}

Furthermore, note that
\begin{align*}
C_k
&=R_1(-k)\cdot\bigl(R_2(t)(t+k)\bigr)\big|_{t=-k}
\\
&=R_1(-k)\cdot(-1)^{k-a_4}\binom{b_4-a_4-1}{k-a_4}
\quad\text{for}\; k\in\mathbb Z,
\end{align*}
and the expression in fact vanishes if $k$ is outside the range $a_4^*\le k\le b_4-1$.

For $\ell\in\mathbb Z$ we have
\begin{align*}
&
\frac{\d}{\d t}\bigl(R(t)(t+\ell)\bigr)\bigg|_{t=-\ell}
=\frac{\d}{\d t}\bigl(R_1(t)\cdot R_2(t)(t+\ell)\bigr)\bigg|_{t=-\ell}
\\ &\quad
=\frac{\d R_1(t)}{\d t}\bigg|_{t=-\ell}
\cdot\bigl(R_2(t)(t+\ell)\bigr)\big|_{t=-\ell}
+R_1(-\ell)\cdot\frac{\d}{\d t}\bigr(R_2(t)(t+\ell)\bigr)\bigg|_{t=-\ell}
\displaybreak[2]\\ &\quad
=\frac{\d R_1(t)}{\d t}\bigg|_{t=-\ell}\cdot(-1)^{\ell-a_4}\binom{b_4-a_4-1}{\ell-a_4}
\\ &\quad\qquad
+R_1(-\ell)\cdot\frac{\d}{\d t}\sum_{k=a_4}^{b_4-1}(-1)^{k-a_4}\binom{b_4-a_4-1}{k-a_4}
\biggl(1-\frac{-\ell+k}{t+k}\biggr)\bigg|_{t=-\ell}
\displaybreak[2]\\ &\quad
=\frac{\d R_1(t)}{\d t}\bigg|_{t=-\ell}\cdot(-1)^{\ell-a_4}\binom{b_4-a_4-1}{\ell-a_4}
+R_1(-\ell)\sum_{\substack{k=a_4\\k\ne\ell}}^{b_4-1}\frac{(-1)^{k-a_4}\binom{b_4-a_4-1}{k-a_4}}{-\ell+k}
\end{align*}
and
\begin{align*}
\frac{\d}{\d t}\biggl(\sum_{k=a_4^*}^{b_4-1}\frac{C_k}{t+k}\cdot(t+\ell)\biggr)\bigg|_{t=-\ell}
&=\frac{\d}{\d t}\biggl(\sum_{k=a_4}^{b_4-1}\frac{C_k}{t+k}\cdot(t+\ell)\biggr)\bigg|_{t=-\ell}
\\
&=\frac{\d}{\d t}\sum_{k=a_4}^{b_4-1}C_k\biggl(1-\frac{-\ell+k}{t+k}\biggr)\bigg|_{t=-\ell}
=\sum_{\substack{k=a_4\\k\ne\ell}}^{b_4-1}\frac{C_k}{-\ell+k}
\\
&=\sum_{\substack{k=a_4\\k\ne\ell}}^{b_4-1}\frac{R_1(-k)\cdot(-1)^{k-a_4}\binom{b_4-a_4-1}{k-a_4}}{-\ell+k}.
\end{align*}
Therefore,
\begin{align*}
P(-\ell)
&=\frac{\d}{\d t}\bigl(P(t)(t+\ell)\bigr)\big|_{t=-\ell}
=\frac{\d}{\d t}\biggl(R(t)(t+\ell)
-\sum_{k=a_4^*}^{b_4-1}\frac{C_k}{t+k}\cdot(t+\ell)\biggr)\bigg|_{t=-\ell}
\\
&=\frac{\d R_1(t)}{\d t}\bigg|_{t=-\ell}\cdot(-1)^{\ell-a_4}\binom{b_4-a_4-1}{\ell-a_4}
\\ &\qquad
+\sum_{\substack{k=a_4\\k\ne\ell}}^{b_4-1}(-1)^{k-a_4}\binom{b_4-a_4-1}{k-a_4}\frac{R_1(-\ell)-R_1(-k)}{-\ell+k},
\end{align*}
and this implies, on the basis of the inclusions \eqref{eq:P3} above, that
$D_cP(-\ell)\in\mathbb Z$ for all $\ell\in\mathbb Z$.
\end{proof}

Finally, define the quantity
\begin{equation}
r(\ba,\bb)
=\frac{(-1)^d}{2\pi i}\int_{C-i\infty}^{C+i\infty}\biggl(\frac\pi{\sin\pi t}\biggr)^2R(\ba,\bb;t)\,\d t,
\label{eq:P4}
\end{equation}
where $C$ is arbitrary from the interval $-a_2^*<C<1-b_2^*$. The definition does not depend on the choice of $C$,
as the integrand does not have singularities in the strip $-a_2^*<\Re t<1-b_2^*$.

\begin{proposition}
\label{prop1}
We have
\begin{equation}
r(\ba,\bb)=q(\ba,\bb)\zeta(2)-p(\ba,\bb),
\qquad\text{with}\quad
q(\ba,\bb)\in\mathbb Z, \quad D_{c_1}D_{c_2}p(\ba,\bb)\in\mathbb Z,
\label{eq:P5}
\end{equation}
where
\begin{equation*}
c_1=\max\{a_1-b_1,a_2-b_2,a_3-b_3,b_4-a_2^*-1\}
\quad\text{and}\quad
c_2=\max\{d+1,b_4-a_2^*-1\}.
\end{equation*}
In addition,
\begin{align}
q(\ba,\bb)
&=(-1)^{b_4-a_4^*-1}\binom{a_4^*-b_1}{a_4^*-a_1}\binom{a_4^*-b_2}{a_4^*-a_2}\binom{a_4^*-b_3}{a_4^*-a_3}\binom{b_4-a_4-1}{a_4^*-a_4}
\nonumber\\ &\quad\times
{}_4F_3\biggl(\begin{matrix} -(b_4-a_4^*-1), \, a_4^*-b_1+1, \, a_4^*-b_2+1, \, a_4^*-b_3+1 \\[2pt]
a_4^*-a_1^*+1, \, a_4^*-a_2^*+1, \, a_4^*-a_3^*+1 \end{matrix}\biggm|1\biggr),
\label{eq:P6}
\end{align}
and the quantity $r(\ba,\bb)/\Pi(\ba,\bb)$ is invariant under any permutation
of the parameters $a_1,a_2,a_3,a_4$.
\end{proposition}

\begin{proof}
We choose $C=1/2-a_2^*$ in~\eqref{eq:P4} and write \eqref{eq:P1} as
\begin{equation*}
R(t)=\sum_{k=a_4^*}^{b_4-1}\frac{C_k}{t+k}+\sum_{\ell=0}^dA_\ell P_\ell(t+a_2^*),
\end{equation*}
where
\begin{equation*}
P_\ell(t)=\frac{(t-1)(t-2)\dotsb(t-\ell)}{\ell!}
\end{equation*}
and $D_cA_\ell\in\mathbb Z$ in accordance with Lemma~\ref{lem:ap}.
Applying Lemma~\ref{barnes} we obtain
\begin{align*}
r(\ba,\bb)
&=\frac{(-1)^d}{2\pi i}\int_{1/2-i\infty}^{1/2+i\infty}\biggl(\frac\pi{\sin\pi t}\biggr)^2R(t-a_2^*)\,\d t
\\
&=(-1)^d\sum_{m=1-a_2^*}^\infty\sum_{k=a_4^*}^{b_4-1}\frac{C_k}{(m+k)^2}
+\sum_{\ell=0}^d\frac{(-1)^{d+\ell}A_\ell}{\ell+1}
\\
&=\zeta(2)\cdot(-1)^d\sum_{k=a_4^*}^{b_4-1}C_k
-(-1)^d\sum_{k=a_4^*}^{b_4-1}C_k\sum_{\ell=1}^{k-a_2^*}\frac1{\ell^2}
+\sum_{\ell=0}^d\frac{(-1)^{d+\ell}A_\ell}{\ell+1}.
\end{align*}
This representation clearly implies that $r(\ba,\bb)$ has the desired form~\eqref{eq:P5},
while the hypergeometric form \eqref{eq:P6} follows from
$$
q(\ba,\bb)=(-1)^d\sum_{k=a_4^*}^{b_4-1}C_k
$$
and the explicit formula \eqref{eq:P2} for $C_k$.
Finally, the invariance of $r(\ba,\bb)/\Pi(\ba,\bb)$ under permutations
of $a_1,a_2,a_3,a_4$ follows from~\eqref{eq:P0} and definition~\eqref{eq:P4} of $r(\ba,\bb)$.
\end{proof}

Assume that the parameters $(\ba,\bb)$ are chosen in the following way:
\begin{equation}
\begin{alignedat}{4}
a_1&=\alpha_1n+1, &\quad a_2&=\alpha_2n+1, &\quad a_3&=\alpha_3n+1, &\quad a_4&=\alpha_4n+1,
\\
b_1&=\beta_1n+1, &\quad b_2&=\beta_2n+1, &\quad b_3&=\beta_3n+1, &\quad b_4&=\beta_4n+2,
\end{alignedat}
\label{P-gen}
\end{equation}
where the \emph{fixed} integers $\alpha_j$ and $\beta_j$, $j=1,\dots,4$, satisfy
$$
\begin{gathered}
\beta_1,\beta_2,\beta_3<\alpha_1,\alpha_2,\alpha_3,\alpha_4<\beta_4,
\\
\alpha_1+\alpha_2+\alpha_3+\alpha_4>\beta_1+\beta_2+\beta_3+\beta_4.
\end{gathered}
$$
The quantities \eqref{eq:P5} in these settings become dependent on a single parameter $n=0,1,2,\dots$,
so we let $r_n=r(\ba,\bb)$, $q_n=q(\ba,\bb)$, $p_n=p(\ba,\bb)$ and identify the characteristics
$c_1=\gamma_1n$ and $c_2=\gamma_2n$ of Proposition~\ref{prop1}, where $\gamma_1$ and $\gamma_2$
are completely determined by $\alpha_j$ and $\beta_j$, $j=1,\dots,4$. The statement below
is proven by standard techniques and is very similar to \cite[Lemmas~10--12]{Zu04}.

\begin{proposition}
\label{prop1compl}
In the above notation,
let $\tau_0$, $\ol{\tau_0}\in\mathbb{C}\setminus\mathbb{R}$
and $\tau_1\in\mathbb R$ be the zeroes of the cubic polynomial
$\prod_{j=1}^4(\tau-\alpha_j)-\prod_{j=1}^4(\tau-\beta_j)$. Define
\begin{align*}
f_0(\tau)
&=\sum_{j=1}^4\bigl(\alpha_j\log(\tau-\alpha_j)-\beta_j\log(\tau-\beta_j)\bigr)
\\ &\qquad
-\sum_{j=1}^3(\alpha_j-\beta_j)\log(\alpha_j-\beta_j)+(\beta_4-\alpha_4)\log(\beta_4-\alpha_4).
\end{align*}
Then
$$
\limsup_{n\to\infty}\frac{\log|r_n|}n=\Re f_0(\tau_0)
\quad\text{and}\quad
\lim_{n\to\infty}\frac{\log|q_n|}n=\Re f_0(\tau_1).
$$
Furthermore,
$$
\Phi_n^{-1}q_n,\,\Phi_n^{-1}D_{\gamma_1n}D_{\gamma_2n}p_n\in\mathbb Z
$$
with
$$
\Phi_n=\prod_{\substack{\text{$p$ prime}\\p\le\min\{\gamma_1,\gamma_2\}n}}p^{\varphi(n/p)},
$$
where
\begin{align*}
\varphi(x)=\max_{\boldsymbol\alpha'=\sigma\boldsymbol\alpha:\sigma\in\mathfrak S_4}
&\biggl(\lf(\beta_4-\alpha_4)x\rf-\lf(\beta_4-\alpha_4')x\rf
\\ &\qquad
-\sum_{j=1}^3\bigl(\lf(\alpha_j-\beta_j)x\rf-\lf(\alpha_j'-\beta_j)x\rf\bigr)\biggr),
\end{align*}
so that the maximum is taken over all permutations $(\alpha_1',\alpha_2',\alpha_3',\alpha_4')$ of
$(\alpha_1,\alpha_2,\alpha_3,\alpha_4)$, and we have
$$
\lim_{n\to\infty}\frac{\log\Phi_n}n
=\int_0^1\varphi(x)\,\d\psi(x)-\int_0^{1/\min\{\gamma_1,\gamma_2\}}\varphi(x)\,\frac{\d x}{x^2},
$$
where $\psi(x)$ is the logarithmic derivative of the gamma function.
\end{proposition}

Here and in what follows, the notation
$\lf\,\cdot\,\rf$ and $\lc\,\cdot\,\rc$ is used for the floor and ceiling integer-part functions.

\subsection({Construction of linear forms in 1 and \003\266(3)})%
{Construction of linear forms in $1$ and $\zeta(3)$}
\label{s1}

The construction in this subsection depends on another set of integral parameters
\begin{equation*}
(\hat\ba,\hat\bb)
=\biggl(\begin{matrix} \hat a_0, \hat a_1, \, \hat a_2, \, \hat a_3 \\ \hat b_0, \, \hat b_1, \, \hat b_2, \, \hat b_3 \end{matrix}\biggr)
\end{equation*}
which satisfies the conditions
\begin{equation}
\begin{gathered}
\tfrac12\hat b_0,\hat b_1\le\tfrac12\hat a_0,\hat a_1,\hat a_2,\hat a_3<\hat b_2,\hat b_3,
\\
\hat a_0+\hat a_1+\hat a_2+\hat a_3\le\hat b_0+\hat b_1+\hat b_2+\hat b_3-2.
\end{gathered}
\label{cond2}
\end{equation}
To this set we assign the rational function
\begin{align}
\hat R(t)
&=\hat R(\hat\ba,\hat\bb;t)
=\frac{(2t+\hat b_0)(2t+\hat b_0+1)\dotsb(2t+\hat a_0-1)}{(\hat a_0-\hat b_0)!}
\cdot\frac{(t+\hat b_1)\dotsb(t+\hat a_1-1)}{(\hat a_1-\hat b_1)!}
\nonumber\\ &\phantom{=R(\ba,\bb;t)} \qquad\times
\frac{(\hat b_2-\hat a_2-1)!}{(t+\hat a_2)\dotsb(t+\hat b_2-1)}
\cdot\frac{(\hat b_3-\hat a_3-1)!}{(t+\hat a_3)\dotsb(t+\hat b_3-1)}
\label{eq:3}
\\
&=\hat\Pi(\hat\ba,\hat\bb)\cdot\frac{\Gamma(2t+\hat a_0)\,\Gamma(t+\hat a_1)\,\Gamma(t+\hat a_2)\,\Gamma(t+\hat a_3)}
{\Gamma(2t+\hat b_0)\,\Gamma(t+\hat b_1)\,\Gamma(t+\hat b_2)\,\Gamma(t+\hat b_3)},
\label{eq:T0}
\end{align}
where
$$
\hat\Pi(\hat\ba,\hat\bb)
=\frac{(\hat b_2-\hat a_2-1)!\,(\hat b_3-\hat a_3-1)!}{(\hat a_0-\hat b_0)!\,(\hat a_1-\hat b_1)!}.
$$
As in \S\,\ref{gc1} we introduce the ordered versions $\hat a_1^*\le\hat a_2^*\le\hat a_3^*$
of the parameters $\hat a_1,\hat a_2,\hat a_3$ and $\hat b_2^*\le\hat b_3^*$ of $\hat b_2,\hat b_3$.
Then this ordering and conditions~\eqref{cond2} imply that $\hat R(t)=O(1/t^2)$ as $t\to\infty$,
the rational function has poles at $t=-k$ for $\hat a_2^*\le k\le\hat b_3^*-1$,
double poles at $t=-k$ for $\hat a_3^*\le k\le\hat b_2^*-1$, and double zeroes
at $t=-\ell$ for $\max\{\lc\hat b_0/2\rc,\hat b_1\}\le\ell\le\min\{\lf(\hat a_0-1)/2\rf,\hat a_1^*-1\}$.

The partial-fraction decomposition of $\hat R(t)$ assumes the form
\begin{equation}
\hat R(t)=\sum_{k=\hat a_3^*}^{\hat b_2^*-1}\frac{A_k}{(t+k)^2}+\sum_{k=\hat a_2^*}^{\hat b_3^*-1}\frac{B_k}{t+k},
\label{eq:T1}
\end{equation}
where
\begin{align}
A_k
&=\bigl(\hat R(t)(t+k)^2\bigr)|_{t=-k}
\nonumber\\
&=(-1)^{\hat d}\binom{2k-\hat b_0}{2k-\hat a_0}\binom{k-\hat b_1}{k-\hat a_1}
\binom{\hat b_2-\hat a_2-1}{k-\hat a_2}\binom{\hat b_3-\hat a_3-1}{k-\hat a_3}\in\mathbb Z
\label{eq:T2}
\end{align}
with $\hat d=\hat a_0+\hat a_1+\hat a_2+\hat a_3-\hat b_0-\hat b_1$,
for $k=\hat a_3^*,\hat a_3^*+1,\dots,\hat b_2^*-1$ and, similarly,
$$
B_k
=\frac{\d}{\d t}\bigl(\hat R(t)(t+k)^2\bigr)|_{t=-k}
$$
for $k=\hat a_2^*,\hat a_2^*+1,\dots,\hat b_3^*-1$. The inclusions
\begin{equation}
D_{\max\{\hat a_0-\hat b_0,\hat a_1-\hat b_1,\hat b_3^*-\hat a_2-1,\hat b_3^*-\hat a_3-1\}}\cdot B_k\in\mathbb Z
\label{eq:T3}
\end{equation}
follow then from standard consideration; see, for example, Lemma~3 and the proof of Lemma~4 in~\cite{Zu04}.
In addition,
\begin{equation}
\sum_{k=\hat a_2^*}^{\hat b_3^*-1}B_k
=-\Res_{t=\infty}\hat R(t)=0
\label{eq:T4}
\end{equation}
by the residue sum theorem.

The quantity of our interest in this section is
\begin{equation}
\hat r(\hat\ba,\hat\bb)
=\frac{(-1)^{\hat d}}{4\pi i}\int_{C-i\infty}^{C+i\infty}\biggl(\frac\pi{\sin\pi t}\biggr)^2\hat R(\hat\ba,\hat\bb;t)\,\d t,
\label{eq:T5}
\end{equation}
where $C$ is arbitrary from the interval $-\min\{\hat a_0/2,\hat a_1^*\}<C<1-\max\{\hat b_0/2,\hat b_1\}$.

\begin{proposition}
\label{prop2}
We have
\begin{equation}
\hat r(\hat\ba,\hat\bb)=\hat q(\hat\ba,\hat\bb)\zeta(3)-\hat p(\hat\ba,\hat\bb),
\qquad\text{with}\quad
\hat q(\hat\ba,\hat\bb)\in\mathbb Z, \quad 2D_{\hat c_1}D_{\hat c_2}^2\hat p(\hat\ba,\hat\bb)\in\mathbb Z,
\label{eq:T6}
\end{equation}
where
\begin{equation*}
\begin{gathered}
\hat c_1=\max\{\hat a_0-\hat b_0,\hat a_1-\hat b_1,\hat b_3^*-\hat a_2-1,\hat b_3^*-\hat a_3-1,\hat b_2^*-\lc\hat a_0/2\rc-1,\hat b_2^*-\hat a_1^*-1\},
\\
\hat c_2=\max\{\hat b_3^*-\lc\hat a_0/2\rc-1,\hat b_3^*-\hat a_1^*-1\}.
\end{gathered}
\end{equation*}
Furthermore,
\begin{align}
&
\hat q(\hat\ba,\hat\bb)
=\binom{2\hat a_3^*-\hat b_0}{2\hat a_3^*-\hat a_0}\binom{\hat a_3^*-\hat b_1}{\hat a_3^*-\hat a_1}
\binom{\hat b_2-\hat a_2-1}{\hat a_3^*-\hat a_2}\binom{\hat b_3-\hat a_3-1}{\hat a_3^*-\hat a_3}
\nonumber\\ &\;\times
{\small
{}_5F_4\biggl(\begin{matrix} -(\hat b_2-\hat a_3^*-1), \, -(\hat b_3-\hat a_3^*-1), \, \hat a_3^*-\hat b_1+1,
\, \hat a_3^*-\tfrac12\hat b_0+\tfrac12, \, \hat a_3^*-\tfrac12\hat b_0+1 \\[3.6pt]
\hat a_3^*-\hat a_1^*+1, \, \hat a_3^*-\hat a_2^*+1,
\, \hat a_3^*-\tfrac12\hat a_0+\tfrac12, \, \hat a_3^*-\tfrac12\hat a_0+1 \end{matrix}\biggm|1\biggr),
}
\label{eq:T7}
\end{align}
and the quantity $\hat r(\hat\ba,\hat\bb)/\hat\Pi(\hat\ba,\hat\bb)$ is invariant under any permutation
of the parameters $\hat a_1,\hat a_2,\hat a_3$.
\end{proposition}

\begin{proof}
Denote $\hat a^*=\min\{\lc\hat a_0/2\rc,\hat a_1^*\}$ and choose $C=1/2-\hat a^*$ in~\eqref{eq:T5} to write
\begin{align*}
\hat r(\hat\ba,\hat\bb)
&=-\frac{(-1)^{\hat d}}2\sum_{m=1-\hat a^*}^\infty\frac{\d \hat{R}(t)}{\d t}\bigg|_{t=m}
\\
&=(-1)^{\hat d}\sum_{m=1-\hat a^*}^\infty\sum_{k=\hat a_3^*}^{\hat b_2^*-1}\frac{A_k}{(m+k)^3}
+\frac{(-1)^{\hat d}}2\sum_{m=1-\hat a^*}^\infty\sum_{k=\hat a_2^*}^{\hat b_3^*-1}\frac{B_k}{(m+k)^2}
\\
&=\zeta(3)\cdot(-1)^{\hat d}\sum_{k=\hat a_3^*}^{\hat b_2^*-1}A_k
\\ &\qquad
-(-1)^{\hat d}\sum_{k=\hat a_3^*}^{\hat b_2^*-1}A_k\sum_{\ell=1}^{k-\hat a^*}\frac1{\ell^3}
-\frac{(-1)^{\hat d}}2\sum_{k=\hat a_2^*}^{\hat b_3^*-1}B_k\sum_{\ell=1}^{k-\hat a^*}\frac1{\ell^2},
\end{align*}
where equality \eqref{eq:T4} was used. In view of the inclusions \eqref{eq:T2}, \eqref{eq:T3}
the found representation of $\hat r(\hat\ba,\hat\bb)$ implies the form~\eqref{eq:T6}.
The hypergeometric form \eqref{eq:T7} follows from
$$
\hat q(\hat\ba,\hat\bb)=(-1)^{\hat d}\sum_{k=\hat a_3^*}^{\hat b_2^*-1}A_k
$$
and the explicit formula \eqref{eq:T2} for $A_k$.
Finally, the invariance of $\hat r(\hat\ba,\hat\bb)/\hat\Pi(\hat\ba,\hat\bb)$
under permutations of $\hat a_1,\hat a_2,\hat a_3$ follows from~\eqref{eq:T0} and
definition~\eqref{eq:T5} of $\hat r(\hat\ba,\hat\bb)$.
\end{proof}

Similar to our choice in \S\,\ref{gc1}, we take the parameters $(\hat\ba,\hat\bb)$ as follows:
\begin{equation}
\begin{alignedat}{4}
\hat a_0&=\hat\alpha_0n+2, &\quad \hat a_1&=\hat\alpha_1n+1, &\quad \hat a_2&=\hat\alpha_2n+1, &\quad \hat a_3&=\hat\alpha_3n+1,
\\
\hat b_0&=\hat\beta_0n+2, &\quad \hat b_1&=\hat\beta_1n+1, &\quad \hat b_2&=\hat\beta_2n+2, &\quad \hat b_3&=\hat\beta_3n+2,
\end{alignedat}
\label{T-gen}
\end{equation}
where the fixed integers $\hat\alpha_j$ and $\hat\beta_j$, $j=0,\dots,3$, satisfy
$$
\begin{gathered}
\tfrac12\hat\beta_0,\hat\beta_1<\tfrac12\hat\alpha_0\hat\alpha_1,\hat\alpha_2,\hat\alpha_3<\hat\beta_2,\hat\beta_3,
\\
\hat\alpha_0+\hat\alpha_1+\hat\alpha_2+\hat\alpha_3=\hat\beta_0+\hat\beta_1+\hat\beta_2+\hat\beta_3;
\end{gathered}
$$
note that the equality is assumed in the latter relation (compare to~\eqref{cond2}) to simplify
the asymptotic consideration in Proposition~\ref{prop2compl}.
The quantities \eqref{eq:T6} then depend on $n=0,1,2,\dots$;
we write $\hat r_n=\hat r(\hat\ba,\hat\bb)$, $\hat q_n=\hat q(\hat\ba,\hat\bb)$, $\hat p_n=\hat p(\hat\ba,\hat\bb)$
and identify the characteristics
$\hat c_1=\hat\gamma_1n$ and $\hat c_2=\hat\gamma_2n$ of Proposition~\ref{prop2}.
Proving the analytical part of the following statement is again similar to what is done in \cite[Lemma~12 or Lemma~20]{Zu04},
while the arithmetic part follows from the results in \cite[Section~7]{Zu04} (cf.\ \cite[Lemma~19]{Zu04}).

\begin{proposition}
\label{prop2compl}
In the above notation,
let $\hat\tau_0$, $\ol{\hat\tau_0}\in\mathbb{C}\setminus\mathbb{R}$ and $\hat\tau_1\in\mathbb R$ be the zeroes of the \emph{cubic}
polynomial
$(\tau-\hat\alpha_0/2)^2\prod_{j=1}^3(\tau-\hat\alpha_j)-(\tau-\hat\beta_0/2)^2\prod_{j=1}^3(\tau-\hat\beta_j)$.
Define
\begin{align*}
\hat f_0(\tau)
&=\hat\alpha_0\log(\tau-\hat\alpha_0/2)-\hat\beta_0\log(\tau-\hat\beta_0/2)
+\sum_{j=1}^3\bigl(\hat\alpha_j\log(\tau-\hat\alpha_j)-\hat\beta_j\log(\tau-\hat\beta_j)\bigr)
\\ &\qquad
-(\hat\alpha_0-\hat\beta_0)\log(\hat\alpha_0/2-\hat\beta_0/2)-(\hat\alpha_1-\hat\beta_1)\log(\hat\alpha_1-\hat\beta_1)
\\ &\qquad
+(\hat\beta_2-\hat\alpha_2)\log(\hat\beta_2-\hat\alpha_2)+(\hat\beta_3-\hat\alpha_3)\log(\hat\beta_3-\hat\alpha_3).
\end{align*}
Then
$$
\limsup_{n\to\infty}\frac{\log|\hat r_n|}n=\Re\hat f_0(\hat\tau_0)
\quad\text{and}\quad
\lim_{n\to\infty}\frac{\log|\hat q_n|}n=\Re\hat f_0(\hat\tau_1).
$$
Furthermore,
$$
\hat\Phi_n^{-1}\hat q_n,\,2\hat\Phi_n^{-1}D_{\hat\gamma_1n}D_{\hat\gamma_2n}^2\hat p_n\in\mathbb Z
$$
with
$$
\hat\Phi_n=\prod_{\substack{\text{$p$ prime}\\p\le\min\{\hat\gamma_1,\hat\gamma_2\}n}}p^{\hat\varphi(n/p)},
$$
where
\begin{align*}
\hat\varphi(x)=\min_{0\le y<1}\biggl(
&
\lf 2y-\hat\beta_0x\rf-\lf 2y-\hat\alpha_0x\rf-\lf(\hat\alpha_0-\hat\beta_0)x\rf
\\ &\qquad
+\lf y-\hat\beta_1x\rf-\lf y-\hat\alpha_1x\rf-\lf(\hat\alpha_1-\hat\beta_1)x\rf
\displaybreak[2]\\ &\qquad
+\lf(\hat\beta_2-\hat\alpha_2)x\rf-\lf\hat\beta_2x-y\rf-\lf y-\hat\alpha_2x\rf
\\ &\qquad
+\lf(\hat\beta_3-\hat\alpha_3)x\rf-\lf\hat\beta_3x-y\rf-\lf y-\hat\alpha_3x\rf\biggr),
\end{align*}
so that we have
$$
\lim_{n\to\infty}\frac{\log\hat\Phi_n}n
=\int_0^1\hat\varphi(x)\,\d\psi(x)-\int_0^{1/\min\{\hat\gamma_1,\hat\gamma_2\}}\hat\varphi(x)\,\frac{\d x}{x^2}.
$$
\end{proposition}


\section({Simultaneous diophantine properties of \003\266(2) and \003\266(3)})%
{Simultaneous diophantine properties of $\zeta(2)$ and $\zeta(3)$}
\label{z23}


In this section we prove Theorem \ref{th} stated in the introduction by combining the constructions of \S\,\ref{gc1} and \S\,\ref{s1}.

\begin{construction}
\label{ex:P}
If we specialize the set of parameters $(\ba,\bb)$ of \S\,\ref{gc1} to be
\begin{equation}
\begin{alignedat}{4}
a_1&=8n+1, &\quad a_2&=7n+1, &\quad a_3&=10n+1, &\quad a_4&=\phantom09n+1,
\\
b_1&=1, &\quad b_2&=\phantom1n+1, &\quad b_3&=2n+1, &\quad b_4&=15n+2,
\end{alignedat}
\label{P-ex}
\end{equation}
then Propositions~\ref{prop1} and \ref{prop1compl} imply that
\begin{equation}
r_n=r(\ba,\bb)=q_n\zeta(2)-p_n,
\qquad\text{where}\quad
\Phi_n^{-1}q_n, \, \Phi_n^{-1}D_{8n}D_{16n}p_n\in\mathbb Z,
\label{P-fin}
\end{equation}
and
\begin{equation}
q_n=\frac{(-1)^n\,(9n)!\,(10n)!}{n!\,(2n)!\,(3n)!\,(5n)!\,(8n)!}
{}_4F_3\biggl(\begin{matrix} -5n, \, 10n+1, \, 9n+1, \, 8n+1 \\[2pt]
3n+1, \, 2n+1, \, n+1 \end{matrix}\biggm|1\biggr).
\label{P-hyper}
\end{equation}
The corresponding function $\varphi(x)$ which defines $\Phi_n$ is
$$
\varphi(x)=\begin{cases}
1 &\text{if } x\in\bigl[\frac1{10},\frac19\bigr)\cup\bigl[\frac17,\frac29\bigr)\cup\bigl[\frac27,\frac13\bigr)\cup\bigl[\frac25,\frac12\bigr)
\cup\bigl[\frac59,\frac47\bigr)\cup\bigl[\frac23,\frac57\bigr)\cup\bigl[\frac45,\frac67\bigr), \\
2 &\text{if } x\in\bigl[\frac19,\frac18\bigr)\cup\bigl[\frac29,\frac14\bigr)\cup\bigl[\frac13,\frac38\bigr)\cup\bigl[\frac47,\frac58\bigr)
\cup\bigl[\frac57,\frac34\bigr)\cup\bigl[\frac67,\frac78\bigr), \\
0 &\text{otherwise},
\end{cases}
$$
so that
$$
\lim_{n\to\infty}\frac{\log\Phi_n}n
=6.61268356\hdots,
$$
and the growth of $r_n$ and $q_n$ as $n\to\infty$ is determined by
$$
\limsup_{n\to\infty}\frac{\log|r_n|}n=-19.10095491\hdots
\quad\text{and}\quad
\lim_{n\to\infty}\frac{\log|q_n|}n=27.86755317\dotsc.
$$
\end{construction}

\begin{construction}
\label{ex:T}
If we specialize the set of parameters $(\hat\ba,\hat\bb)$ of \S\,\ref{s1} to be
\begin{equation}
\begin{alignedat}{4}
\hat a_0&=16n+2, &\quad \hat a_1&=8n+1, &\quad \hat a_2&=\phantom09n+1, &\quad \hat a_3&=10n+1,
\\
\hat b_0&=11n+2, &\quad \hat b_1&=1, &\quad \hat b_2&=16n+2, &\quad \hat b_3&=16n+2,
\end{alignedat}
\label{T-ex}
\end{equation}
we obtain from Propositions~\ref{prop2} and \ref{prop2compl} that
\begin{equation}
\hat r_n=\hat r(\hat\ba,\hat\bb)=\hat q_n\zeta(3)-\hat p_n,
\qquad\text{where}\quad
\hat\Phi_n^{-1}\hat q_n, \, 2\hat\Phi_n^{-1}D_{8n}^3\hat p_n\in\mathbb Z,
\label{T-fin}
\end{equation}
and
\begin{equation}
\hat q_n=\frac{(7n)!\,(9n)!\,(10n)!}{n!\,(2n)!\,(4n)!\,(5n)!\,(6n)!\,(8n)!}\,
{}_5F_4\biggl(\begin{matrix} -6n, \, -6n, \, 10n+1, \, \tfrac92n+\tfrac12, \, \tfrac92n+1 \\[2.5pt]
2n+1, \, n+1, \, 2n+\tfrac12, \, 2n+1 \end{matrix}\biggm|1\biggr).
\label{T-hyper}
\end{equation}
The corresponding function $\hat\varphi(x)$ assumes the form
$$
\hat\varphi(x)=\begin{cases}
1 &\text{if } x\in\bigl[\frac1{10},\frac18\bigr)\cup\bigl[\frac17,\frac14\bigr)\cup\bigl[\frac27,\frac13\bigr)\cup\bigl[\frac37,\frac12\bigr)
\cup\bigl[\frac59,\frac47\bigr)\cup\bigl[\frac35,\frac58\bigr)\cup\bigl[\frac23,\frac57\bigr)\cup\bigl[\frac56,\frac67\bigr), \\
2 &\text{if } x\in\bigl[\frac13,\frac38\bigr)\cup\bigl[\frac47,\frac35\bigr)\cup\bigl[\frac57,\frac34\bigr)\cup\bigl[\frac67,\frac78\bigr), \\
0 &\text{otherwise},
\end{cases}
$$
so that
$$
\lim_{n\to\infty}\frac{\log\hat\Phi_n}n
=\vphi=5.70169601\hdots,
$$
and the growth of $\hat r_n$ and $\hat q_n$ as $n\to\infty$ is determined by
$$
\limsup_{n\to\infty}\frac{\log|\hat r_n|}n=-\rho=-19.10095491\hdots
\quad\text{and}\quad
\lim_{n\to\infty}\frac{\log|\hat q_n|}n=\kappa=27.86755317\hdots\,
$$
with the same letters $\vphi$, $\kappa$ and $\rho$ as in the introduction.
\end{construction}

\begin{connection}
Surprisingly\,---\,and this could be guess\-ed from the asymptotics above,
the coefficients in \eqref{P-fin} of $\zeta(2)$ and in \eqref{T-fin} of $\zeta(3)$
coincide: $q_n=\hat q_n$. This follows from the following classical identity\,---\,Whipple's transformation \cite[p.~65, eq.~(2.4.2.3)]{Sl},
in which we assume that $b=-N$ is a negative integer:
\begin{multline}
{}_4F_3\biggl(\begin{matrix} f, \, 1+f-h, \, h-a, \, b \\
h, \, 1+f+a-h, \, g \end{matrix}\biggm|1\biggr)
=\frac{(g-f)_N}{(g)_N}
\\ \times
{}_5F_4\biggl(\begin{alignedat}{5}
a&, \, & b&, \, & 1+f-g&, \, & \tfrac12f&, \, & \tfrac12f+\tfrac12 \\
&& h&, \, & 1+f+a-h&, \, & \tfrac12(1+f+b-g)&, \, & \tfrac12(1+f+b-g)+\tfrac12 \\
\end{alignedat}\biggm|1\biggr).
\label{whipple}
\end{multline}
The particular choices \eqref{P-ex} and \eqref{T-ex} correspond to taking
$a=b=-6n$, $f=9n+1$, $h=n+1$ and $g\to-n+1$ in~\eqref{whipple}.
The equality $q_n=\hat{q}_n$ can be alternatively established by examining the recurrence equation
satisfied by both $q_n$ and $\hat q_n$; we outline the equation in our proof of Theorem~\ref{th} below.

Note that we also have $\Phi_n$ divisible by $\hat\Phi_n$ in the construction above, so
that we can `merge' the corresponding arithmetic properties \eqref{P-fin} and \eqref{T-fin} as follows:
\begin{equation}
\hat\Phi_n^{-1}q_n, \, \hat\Phi_n^{-1}D_{8n}D_{16n}p_n, \, 2\hat\Phi_n^{-1}D_{8n}^3\hat p_n\in\mathbb Z.
\label{PT-fin}
\end{equation}
In both situations we get
$$
\lim_{n\to\infty}\frac{\log(\hat\Phi_n^{-1}D_{8n}D_{16n})}n
=\lim_{n\to\infty}\frac{\log(2\hat\Phi_n^{-1}D_{8n}^3)}n
=24-\vphi
=18.29830398\hdots
$$
and
$$
\lim_{n\to\infty}\frac{\log|\hat q_n|}n=\kappa=27.86755317\dotsc,
$$
so that both families of rational approximations to $\zeta(2)$ and $\zeta(3)$ are diophantine:
\begin{align*}
\limsup_{n\to\infty}\frac{\log|\hat\Phi_n^{-1}D_{8n}D_{16n}r_n|}n
&=\limsup_{n\to\infty}\frac{\log|2\hat\Phi_n^{-1}D_{8n}^3\hat r_n|}n
\\
&=24-\vphi-\rho
=-0.80265093\hdots
<0.
\end{align*}
\end{connection}


\begin{proof}[Proof of Theorem \textup{\ref{th}}]
Using the notation above we define $\tau_0$ and $s_0$ in accordance with~\eqref{eq-C}.

To prove the theorem, we use a recurrence relation satisfied by $q_n$, $p_n$ and $\hat{p}_n$.
We execute the Gosper--Zeilberger algorithm of creative telescoping separately for
the rational function $R_n(t)=R(t)$ defined in \eqref{eq:gc} and specialised by \eqref{P-ex},
and for $\hat R_n(t)=\hat R(t)$ defined in~\eqref{eq:3} with the choice of parameters \eqref{T-ex}.
The results in both cases are polynomials $P_0(n),\dots\,,P_3(n)\in\Z[n]$ and rational functions
$S_n(t),\hat S_n(t)$ such that
\begin{align*}
P_3(n)R_{n+3}(t)+P_2(n)R_{n+2}(t)+P_1(n)R_{n+1}(t)+P_0(n)R_n(t)
&=S_n(t+1)-S_n(t),
\\
P_3(n)\hat R_{n+3}(t)+P_2(n)\hat R_{n+2}(t)+P_1(n)\hat R_{n+1}(t)+P_0(n)\hat R_n(t)
&=\hat S_n(t+1)-\hat S_n(t).
\end{align*}
Applying then the argument as in the proof of Theorem~5.4 in~\cite{BBBC} we find out that both
the hypergeometric integrals
$$
r_n=\frac1{2\pi i}\int_{-i\infty}^{i\infty}\biggl(\frac\pi{\sin\pi t}\biggr)^2R_n(t)\,\d t
\quad\text{and}\quad
\hat r_n=\frac1{4\pi i}\int_{-i\infty}^{i\infty}\biggl(\frac\pi{\sin\pi t}\biggr)^2\hat R_n(t)\,\d t
$$
satisfy the \emph{same} recurrence equation
$$
P_3(n)y_{n+3}+P_2(n)y_{n+2}+P_1(n)y_{n+1}+P_0(n)y_n=0.
$$
Since $r_n=q_n\zeta(2)-p_n$, $\hat r_n=q_n\zeta(3)-\hat p_n$ and both $\zeta(2)$ and $\zeta(3)$ are irrational,
we deduce that the coefficients $q_n$, $p_n$ and $\hat p_n$ satisfy the same equation.
Using this fact we obtain that the sequence of determinants
$$
\Delta_n=\begin{vmatrix}
q_n & q_{n+1} & q_{n+2} \\
p_n & p_{n+1} & p_{n+2} \\
\hat{p}_n & \hat{p}_{n+1} & \hat{p}_{n+2}
\end{vmatrix}
$$
satisfies the recurrence equation $P_3(n)\Delta_{n+1}+P_0(n)\Delta_n=0$.
The coefficients of $P_3(n)$ are all positive, while the coefficients of $P_0(n)$ are all negative;
the details of this computation can be found on the webpage \cite{these} of the first author. This implies that the nonvanishing
of $\Delta_n$ for some $n$ is equivalent to the nonvanishing of $\Delta_0$.
We have explicitly
\begin{alignat*}{3}
q_0 &= 1, &\quad
q_1 &= 12307565655, &\quad
q_2 &= 5669931265166541788415,
\displaybreak[2]\\
p_0 &= 0, &\quad
p_1 &= \tfrac{199536684432021}{9856}, &\quad
p_2 &= \tfrac{6500408024275547867356589727409007}{696970391040},
\displaybreak[2]\\
\hat{p_0} &= 0, &\quad
\hat{p}_1 &= \tfrac{7953492001094261}{537600}, &\quad
\hat{p}_2 &= \tfrac{37762843816152998347068580008855083}{5540664729600},
\end{alignat*}
so that
$$
\Delta_0 = \begin{vmatrix}
q_0 & q_1 & q_2 \\
p_0 & p_1 & p_2 \\
\hat{p}_0 & \hat{p}_1 & \hat{p}_2
\end{vmatrix}
=\tfrac{288666665737256181552839214834819523}{107268868422523551744000}
\ne0\,.
$$
Thus, $\Delta_n\ne0$ for any $n\geq 0$.

Now let $\ve$, $\eta>0$; for simplicity we may assume $\eta\leq\ve$. Let $m$ be a sufficiently large integer as in the statement of Theorem~\ref{th}.
Let $a_0,a_1,a_2$ satisfy the hypotheses in Theorem~\ref{th}. We take $n=\lc m/8\rc$, so that $8n-7\leq m\leq 8n$. Since
the determinant $\Delta_n$ does not vanish, there exists an $\ell\in\{n\,,n+1\,,n+2\}$ such that
$$
a_0q_{\ell}+a_1p_{\ell}+a_2\hat{p}_{\ell}\ne0\,.
$$
Now we have $m\leq 8n\leq 8\ell$, so that $D_{8\ell}^2D_{16\ell}a_0\in\Z$ and $D_{8\ell}a_1\in\Z$.
Letting $e_{m,\ell}=\frac{2D_{8\ell}}{D_m}$, we get the property
$$
\frac{D_{2m}}{D_m}\biggm|e_{m,l}\,\frac{D_{16\ell}}{2D_{8\ell}},
$$
so that $e_{m,\ell}\dfrac{D_{16\ell}}{2D_{8\ell}}a_2\in\Z$.
Therefore, using the arithmetic properties of $q_{\ell}$, $p_{\ell}$ and $\hat{p}_{\ell}$ we conclude that
\begin{equation}
e_{m,\ell}(D_{8\ell}^2D_{16\ell}a_0)(\Phi_{\ell}^{-1}q_{\ell})
+ e_{m,\ell}(D_{8\ell}a_1)(\Phi_{\ell}^{-1}D_{8\ell}D_{16\ell}p_{\ell})
+ e_{m,\ell}\biggl(\frac{D_{16\ell}}{2D_{8\ell}}a_2\biggr)(2\Phi_{\ell}^{-1}D_{8\ell}^3\hat{p}_{\ell})
\label{integer}
\end{equation}
is a nonzero integer. Note that $\ell\leq\frac{m}{8}+3$, so that the asymptotic
contribution of $e_{m,\ell}$ is almost invisible: $e_{m,\ell}\leq \frac{2D_{m+24}}{D_m}=e^{o(m)}=e^{o(\ell)}$.

Let us bound the integer \eqref{integer} from above. Writing hypothesis (ii) as
$$
|a_0+a_1\zeta(2)+a_2\zeta(3)|\leq e^{-(s_0+\eta)m}\leq e^{-(32-\vphi+\kappa+8\eta)(n-1)},
$$
we obtain
\begin{align*}
&
|a_0q_{\ell}+a_1p_{\ell}+a_2\hat{p}_{\ell}|
\\ &\quad
\leq |q_{\ell}|\,|a_0+a_1\zeta(2)+a_2\zeta(3)|+|a_1|\,|q_{\ell}\zeta(2)-p_{\ell}|+|a_2|\,|q_{\ell}\zeta(3)-\hat{p}_{\ell}|
\\ &\quad
\leq e^{-(32-\vphi+8\ve)n+o(n)},
\end{align*}
since $\ve\leq\eta$. On the other hand, the common denominator of the coefficients used above is
$$
e_{m,\ell}D_{8\ell}^2D_{16\ell}\Phi_{\ell}^{-1}\leq e^{(2\cdot8+16-\vphi)\ell+o(\ell)}=e^{(32-\vphi)n+o(n)}.
$$
This means that the non-zero integer \eqref{integer} has absolute value at most $e^{-8\ve n+o(n)}$,
which is not possible for a sufficiently large~$n$,
thus implying the truth of Theorem~\ref{th}.
\end{proof}


\section{A new diophantine exponent}
\label{NDE}


\subsection{Definition and basic properties}\label{Definition and basic properties}

We now introduce a new exponent that depends on some $\tau\in\R$ and is related to Theorem~\ref{th}.

\begin{Def}\label{def exponent}
Let $\xi_1$, $\xi_2\in\R$ and $\tau\in\R$. We denote by $s_\tau(\xi_1,\xi_2)$ the infimum of the set $E_\tau(\xi_1,\xi_2)$ of all $s\in\R$
with the following property.
Let $\ve>0$ and $n$ be sufficiently large in terms of $\ve$. Let $(a_0,a_1,a_2)\in\Q^3\setminus\{\bold0\}$ be such that:
\begin{enumerate}
  \item[(i)] $D_{n}^2D_{2n}a_0\in\Z$, $D_{n}a_1\in\Z$ and $\dfrac{D_{2n}}{D_{n}}\,a_2\in\Z$; and
  \item[(ii)] $|a_0|,|a_1|,|a_2|$ are bounded from above by $e^{-(\tau+\ve)n}$.
\end{enumerate}
Then $|a_0+a_1\xi_1+a_2\xi_2|>e^{-sn}$.

By convention, we set $s_\tau(\xi_1,\xi_2)=+\infty$ if $E_\tau(\xi_1,\xi_2)=\emptyset$, and $s_\tau(\xi_1,\xi_2)=-\infty$ if $E_\tau(\xi_1,\xi_2)=\R$.
\end{Def}

This definition allows us to restate Theorem \ref{th} as follows.

\begin{theorem}\label{majoration s}
With $\tau_0=0.899668635\dots$ and $s_0=6.770732145\dots$ as in \eqref{eq-C},
we have $s_{\tau_0}(\zeta(2),\zeta(3))\leq s_0$.
\end{theorem}

To begin with, let us state and prove general results on this diophantine exponent $s_\tau(\xi_1,\xi_2)$
depending on the range when $\tau$ varies; it
turns out that it carries diophantine information on $\xi_1$ and $\xi_2$ only if $\tau<1$.


\begin{proposition}\label{propriete de base}
\begin{enumerate}
\item \label{s=-infty}
If $\tau> 4$, then $s_\tau(\xi_1,\xi_2)=-\infty$.
\item \label{s=4}
If $1\leq\tau\leq 4$, then $s_\tau(\xi_1,\xi_2)=4$.
\item \label{minoration s}
If $\tau<1$, then $s_\tau(\xi_1,\xi_2)\geq 6-2\tau$.
\item \label{s(rationnel)}
If $\tau<1$ and at least one of $\xi_1$ or $\xi_2$ is rational, then $s_\tau(\xi_1,\xi_2)=+\infty$.
\item \label{s(indpdce Q-lineaire)}
If $\tau<0$ and the numbers $1$, $\xi_1$ and $\xi_2$ are linearly dependent over $\Q$, then $s_\tau(\xi_1,\xi_2)=+\infty$.
\item \label{s_tau>s_tau'}
If $\tau\leq\tau'$, then $s_{\tau}(\xi_1,\xi_2)\geq s_{\tau'}(\xi_1,\xi_2)$.
\end{enumerate}
\end{proposition}

\begin{proof}
(1) We see that whenever the coefficient $a_i$ is not zero,
we must have $|a_i|\geq1/(D_{n}^2D_{2n})=e^{-4n+o(n)}$ if $i=0$, and an even larger estimate from below
(namely, $e^{-n+o(n)}$) if $i=1$ or $2$. Therefore,
having at least one triple $(a_0,a_1,a_2)\in\Q^3\setminus\{\bold0\}$ that satisfies both (i) and (ii) of Definition~\ref{def exponent}
means $\tau\leq 4$; having no such triple implies
$E_\tau(\xi_1,\xi_2)=\R$.

(2) Assuming now $1\leq\tau\leq 4$ in Definition~\ref{def exponent} and choose $n$ sufficiently large
to accommodate $D_n<e^{(1+\ve)n}$ and $D_{2n}/D_n<e^{(1+\ve)n}$. Condition (ii) implies
that $|a_1|\leq e^{-(\tau+\ve)n}\leq e^{-n-\ve n}$, so that the integer $|D_na_1|\le D_ne^{-n-\ve n}<1$
must be zero, $a_1=0$. Similar consideration shows that $a_2=0$, hence the only nonzero element
in the triple $(a_0,a_1,a_2)\in\Q^3\setminus\{\bold0\}$ is $a_0$.
Then condition (i) implies that $|a_0|\geq1/(D_{n}^2D_{2n})=e^{-4n+o(n)}$ with the equality possible
by simply taking $a_0=1/(D_{n}^2D_{2n})$. Thus, $s_\tau(\xi_1,\xi_2)=4$ for all $\xi_1,\xi_2$ whenever $4\geq\tau\geq1$.

(3) Take $s<6-2\tau$ and define $\ve=\frac{1}{3}(6-2\tau-s)$, so that $s=6-2\tau-3\ve>\tau+\ve/2$ because of $\tau<1$.
Let $n$ be sufficiently large to have $(D_nD_{2n})^2>e^{(6-\ve)n}=e^{-(s+2\tau+2\ve)n}$ satisfied.
Define the set
$$
K=\{(x_0,x_1,x_2)\in\R^3:|x_1|,|x_2|\leq e^{-(\tau+\ve)n},\ |x_0+x_1\xi_1+x_2\xi_2|\leq e^{-sn}\}\subset\mathbb R^3,
$$
which is compact, convex, symmetric with respect to $\bold0$ and has volume $8e^{-(s+2\tau+2\ve)n}$.
Consider the lattice
$$
\Gamma=\frac{1}{D_{n}^2D_{2n}}\Z\oplus\frac{1}{D_{n}}\Z\oplus\frac{D_{n}}{D_{2n}}\Z,
$$
whose fundamental domain has volume
$$
\frac{1}{D_{n}^2D_{2n}}\cdot\frac{1}{D_{n}}\cdot\frac{D_{n}}{D_{2n}}
<e^{-(s+2\tau+2\ve)n}.
$$
By Minkowski's theorem, $K$ contains a nonzero point $(a_0,a_1,a_2)$ of the lattice $\Gamma$, for which we have
\begin{align*}
|a_0|
&\leq |a_1|\,|\xi_1|+|a_2|\,|\xi_2|+|a_0+a_1\xi_1+a_2\xi_2|
\\
&\leq (|\xi_1|+|\xi_2|)e^{-(\tau+\ve)n}+e^{-sn}
\leq e^{-(\tau+\ve/2)n}.
\end{align*}
The estimate means that $s\notin E_\tau(\xi_1,\xi_2)$; as $s_\tau(\xi_1,\xi_2)$ is the infimum of the set $E_\tau(\xi_1,\xi_2)$,
we get $s_\tau(\xi_1,\xi_2)\geq 6-2\tau$.

(4) Assume $\xi_1=p/q\in\Q$, take $\ve\in(0,1-\tau)$.
By choosing $a_0=q\xi_1/D_{n}$, $a_1=-q/D_{n}$ and $a_2=0$ we see that properties (i) and (ii)
in the definition of $E_\tau(\xi_1,\xi_2)$ are satisfied for \emph{any} $n$ sufficiently large.
In addition, $|a_0+a_1\xi_1+a_2\xi_2|=0<e^{-sn}$ for any $s\in\R$, which means that
$E_\tau(\xi_1,\xi_2)=\emptyset$, hence $s_\tau(\xi_1,\xi_2)=+\infty$.

If $\xi_2=p/q\in\Q$, then the choice $a_0=q\xi_2D_{n}/D_{2n}$, $a_1=0$ and $a_2=-qD_{n}/D_{2n}$ does the job.

(5) Assume now that there exist integers $q_0$, $q_1$ and $q_2$, not all zero, such that $p_0+p_1\xi_1+p_2\xi_2=0$.
Setting $a_0=p_0$, $a_1=p_1$ and $a_2=p_2$ we see that properties (i) and (ii) are satisfied with any
choice of $\tau<0$ and $\ve$, for all $n$ sufficiently large in terms of~$\ve$.
At the same time $|a_0+a_1\xi_1+a_2\xi_2|=0<e^{-sn}$ for any $s\in\R$, meaning that $E_\tau(\xi_1,\xi_2)=\emptyset$, hence $s_\tau(\xi_1,\xi_2)=+\infty$.

(6) Using \eqref{s=-infty}, \eqref{s=4} and \eqref{minoration s}, we may assume that $\tau'<1$.
Let $s\in E_\tau(\xi_1,\xi_2)$ meaning that for all $\ve>0$
and for all triples $(a_0,a_1,a_2)\in\Q^3\setminus\{\bold0\}$ which satisfy
$D_{n}^2D_{2n}a_0$, $D_{n}a_1$, $\dfrac{D_{2n}}{D_{n}}a_2\in\Z$ and $|a_0|,|a_1|,|a_2|\leq e^{-(\tau+\ve)n}$, we have
$|a_0+a_1\xi_1+a_2\xi_2|>e^{-sn}$. For $n$ be sufficiently large and $(a_0,a_1,a_2)\in\Q^3\setminus\{\bold0\}$ such that
$D_{n}^2D_{2n}a_0$, $D_{n}a_1$, $\dfrac{D_{2n}}{D_{n}}a_2\in\Z$ and $|a_i|\leq e^{-(\tau'+\ve)n}$, we also
have $|a_i|\leq e^{-(\tau+\ve)n}$. This means that $|a_0+a_1\xi_1+a_2\xi_2|>e^{-sn}$ and $s\in E_{\tau'}(\xi_1,\xi_2)$,
so that $E_{\tau}(\xi_1,\xi_2)\subset E_{\tau'}(\xi_1,\xi_2)$, which leads to claim~\eqref{s_tau>s_tau'} by taking the infimum of both sets.
\end{proof}

From now on we assume $\tau$ to be real $<1$.

\begin{remarks}
Theorem \ref{majoration s} is nontrivial since $\tau_0<1$. However, it does not imply that $1$, $\zeta(2)$
and $\zeta(3)$ are $\Q$-linearly independent since $\tau_0>0$.

Part \eqref{minoration s} of Proposition \ref{propriete de base} yields $s_{\tau_0}(\zeta(2),\zeta(3))\geq4.20$,
so that the statement of Theorem~\ref{majoration s} is far from being best possible.

The fact that $s_{\tau_0}(\zeta(2),\zeta(3))<+\infty$ in Theorem~\ref{majoration s} is already new.
\end{remarks}


\subsection{Omitting one number}\label{Omitting one number}

Recall the definition of the usual exponent of irrationality of $\mu(\xi)$ of a number $\xi\in\R$
from the introductory part.
Here comes its generalisation, the \emph{$\psi$-exponent of irrationality},
given by Fischler in~\cite{Indaga}.

\begin{Def}\label{phi-exponent}
Let $\mathcal{E}$ be the set of all $\psi\colon\N^*\to\N^*$ with the following properties:
for any $q\geq 1$, $\psi(q+1)$ is a multiple of $\psi(q)$, and the limit
$$
\gamma_\psi=\lim_{q\to\infty}\frac{\log\psi(q)}{\log q}
$$
exists and belongs to the interval $[0,1)$.
For $\psi\in\mathcal{E}$ and $\xi\in\R\setminus\Q$, denote by $\mu_\psi(\xi)$ the
supremum of the set $M_\psi(\xi)$
of all $\mu\in\R$ such that there are infinitely many $q\geq 1$ which are divisible by $\psi(q)$ and satisfy
$$
\biggl|\xi-\frac{p}{q}\biggr|\leq \frac{1}{q^\mu} \quad\text{for some}\; p\in\Z\,.
$$
If $M_\psi(\xi)$ is not bounded from above, that is, if $M_\psi(\xi)=\R$, we get $\mu_\psi(\xi)=+\infty$.
\end{Def}

An equivalent way of defining $\mu_\psi(\xi)$, is by letting $\mu_\psi(\xi)$ be the infimum of the set of exponents $\mu$
such that for all $q$ large enough with
$\psi(q)\mid q$ one has $|\xi-p/q|>1/q^{-\mu}$, and taking $\mu_\psi(\xi)=+\infty$ if the set is empty.

When $\psi(q)=1$ for all $q$, the $\psi$-exponent $\mu_{\psi}(\xi)$ coincides with the usual exponent of irrationality $\mu(\xi)$.
It is known \cite[Corollary~3]{Indaga} that $\mu_\psi(\xi)=+\infty$ if and only if $\xi$ is a Liouville number, that is,
$\mu(\xi)=+\infty$. If this is not the case, then
$$
(1-\gamma_\psi)\mu(\xi) \leq \mu_{\psi}(\xi) \leq \mu(\xi).
$$

Fischler proves in \cite{Indaga} that $\mu_{\psi}(\xi)\geq 2-\gamma_{\psi}$ for any $\psi\in\mathcal{E}$ and any $\xi\in\R\setminus\Q$,
with the equality holding for almost all $\xi\in\R$ in the sense of Lebesgue measure.
More precisely, he shows that, given an $\eta>2-\gamma_\psi$, the set of $\xi$ such that
$\mu_{\psi}(\xi)>\eta$ has Hausdorff dimension $(2-\gamma_{\psi})/{\eta}$.

The usual construction of a function $\psi\in\mathcal{E}$ is as follows. One takes $\psi(q)=\delta_n$
with $n=\lf(\log q)/(\delta-\alpha)\rf$, where $(\delta_n)_{n\geq 1}$ is a sequence of positive integers
such that $\delta_n$ divides $\delta_{n+1}$ for each $n\geq 1$ and $\delta_n=e^{\delta n+o(n)}$ as $n\to\infty$, while
$\alpha\in\R$ is chosen to satisfy $\alpha < \delta$. In this construction, we have $\gamma_\psi=\delta/(\delta-\alpha)$.

Definition \ref{phi-exponent} allows us deducing diophantine results involving only quantity, $\xi_1$ or $\xi_2$,
from a nontrivial upper bound for the exponent $s_{\tau}(\xi_1,\xi_2)$ from Definition~\ref{def exponent}.

\begin{proposition}\label{majoration mu_psi}
Let $\xi_1,\xi_2$ be real numbers and $\tau<1$.
Define $\psi_1,\psi_2\colon\N^*\to\N^*$ by taking $\psi_1(q)=D_{n}D_{2n}$ and $\psi_2(q)=D_{n}^3$, where
$n=\lfloor(\log q)/(4-\tau)\rfloor$. Then
$$
\mu_{\psi_i}(\xi_i) \leq \frac{s_\tau(\xi_1,\xi_2)-\tau}{4-\tau}
\quad\text{for}\;\; i=1,2.
$$
\end{proposition}

\begin{proof}
Let $\tau'\in\R$ satisfy $\tau<\tau'<1$. Take $p\in\Z$ and $q\in\N^*$ sufficiently large, $\psi_1(q)\mid q$,
and $m=\lfloor(\log q/(4-\tau')\rfloor$, so that $\psi_1(q)=D_{m}D_{2m}$. We may assume that $|p/q-\xi_1|<1$.

For an $s>s_\tau(\xi_1,\xi_2)$, choose $\ve>0$ such that
$$
\ve<\frac12\,\min\biggl\{\tau'-\tau\,,\frac{(s-4)(\tau'-\tau)}{2s-\tau'-4}\biggr\}\,;
$$
part~\eqref{minoration s} of Proposition \ref{propriete de base} implies $s>4$.

Take
$$
n=\biggl\lfloor\frac{4-\tau'}{4-\tau-2\ve}(m+1)\biggr\rfloor+1<m,
\quad
a_0=\frac{p}{D_{n}^2D_{2n}}
\quad\text{and}\quad
a_1=\frac{-q}{D_{n}^2D_{2n}}.
$$
Then
\[
D_{n}^2D_{2n}a_0\in\Z
\quad\text{and}\quad
D_{n}a_1=\frac{-q}{D_{m}D_{2m}}\,\frac{D_{m}D_{2m}}{D_{n}D_{2n}}
=\frac{-q}{\psi_1(q)}\,\frac{D_{m}D_{2m}}{D_{n}D_{2n}}\in\Z,
\]
and $q<e^{(4-\tau')(m+1)}$ implying that $|a_1| \leq e^{-(\tau+\ve)n}$; for $a_0$ we have $|a_0|=|p|e^{-4n+o(n)}$.
Therefore, $|p|\leq |p-q\xi_1|+|q\xi_1|\leq q(1+|\xi_1|)$, which leads to
$$
|a_0|\leq q(1+|\xi_1|)e^{-4n+o(n)} \leq e^{(4-\tau')(m+1)-(4-\ve)n} \leq e^{-(\tau+\ve)n}
$$
for $q$ sufficiently large.
Letting $a_2=0$ and using $s>s_\tau(\xi_1,\xi_2)$ we deduce that $|a_0+a_1\xi_1|>e^{-sn}$; from the definition of $a_0$ and $a_1$
it follows that
$$
\biggl|\frac{p}{q}-\xi_1\biggr|
>\frac{e^{(4-s-\ve)n}}{q}
$$
provided $q$ is sufficiently large. The assumption on $\ve$ results in the estimate
$$
\biggl|\frac{p}{q}-\xi_1\biggr|
>q^{-(s-\tau')/(4-\tau')}\,,
$$
which implies $\mu_{\psi_1}(\xi_1)\leq (s-\tau')/(4-\tau')$. This upper bound holds for all $s>s_\tau(\xi_1,\xi_2)$;
taking the infimum over~$s$ and then choosing $\tau'\in(\tau,1)$ sufficiently close to $\tau$, completes the proof
for $i=1$. The proof for $i=2$ is similar.
\end{proof}

Since $\mu_{\psi_i}(\xi_i)\geq 2-\gamma_{\psi_i}$ with $\gamma_{\psi_i}=3/(4-\tau)<1$,
Proposition \ref{majoration mu_psi} implies the lower bound $s_\tau(\xi_1,\xi_2)\geq 5-\tau$,
which is however weaker than the one from statement~\eqref{minoration s} of Proposition~\ref{propriete de base}.


\begin{corollary}\label{majoration mu}
For $\xi_1$ and $\xi_2$ real numbers and $\tau<1$,
the following inequalities hold for the ordinary irrationality exponent:
\[
\mu(\xi_i)\leq \frac{s_\tau(\xi_1,\xi_2)-\tau}{1-\tau}
\quad\text{for}\;\; i=1,2.
\]
\end{corollary}

\begin{proof}
In the notation of Proposition~\ref{majoration mu_psi},
use $(1-\gamma_{\psi_i})\mu(\xi_i)\leq\mu_{\psi_i}(\xi_i)$.
\end{proof}


\subsection{Case of linear dependence}\label{Case of linear dependence}

In this subsection, we prove a converse result to Proposition~\ref{majoration mu_psi} above; namely,
under the linear dependence of $1$, $\xi_1$ and $\xi_2$ over~$\Q$,
we deduce an upper bound on $s_\tau(\xi_1,\xi_2)$ from an upper bound on the irrationality exponent of either $\xi_1$ or~$\xi_2$.

\begin{proposition}\label{majoration s_tau si lie}
For $\xi_1,\xi_2\not\in\Q$ assume that $1$, $\xi_1$ and $\xi_2$ are linearly dependent over~$\Q$.
Take $0\leq \tau<1$ and define $\psi\in\mathcal{E}$ by
$\psi(q)=D_{n}^2$ with $n=\lfloor(\log q)/(4-\tau)\rfloor$.
Then $\mu_{\psi}(\xi_1)=\mu_{\psi}(\xi_2)$ and
$$
s_\tau(\xi_1,\xi_2) \leq 4+(\mu_{\psi}(\xi_i)-1)(4-\tau)
\quad\text{for}\;\; i=1,2.
$$
In addition,
$$
s_\tau(\xi_1,\xi_2)\leq 6-\tau
$$
unless both $\xi_1$ and $\xi_2$ belong to a certain set of Lebesgue measure zero.
\end{proposition}

Note that the inequalities of this proposition does not hold if $\tau<0$,
since $s_\tau(\xi_1,\xi_2)=+\infty$ in this case (see Proposition~\ref{propriete de base}).


\begin{proof}
The equality $\mu_{\psi}(\xi_1)=\mu_{\psi}(\xi_2)$ is trivially true for any $\psi\in\mathcal{E}$.

Let $\alpha_0,\alpha_1\in\Q$ be such that $\alpha_0+\alpha_1\xi_1=\xi_2$ and $\psi$ the function defined
in the statement of Proposition \ref{majoration s_tau si lie}. Denote by $A$ a common denominator of $\alpha_0$ and $\alpha_1$.
Take $\ve>0$, $\nu>0$, $\mu>\mu_{\psi}(\xi_1)$ and $n$ be sufficiently large
with respect to $\ve$, $\nu$ and $\mu$. Let $(a_0,a_1,a_2)\in\Q^3\setminus\{\bold0\}$ satisfy $|a_i|\leq e^{-(\tau+\ve)n}$
and $D_{n}^2D_{2n}a_0,D_{n}a_1,\dfrac{D_{2n}}{D_{n}}a_2\in\Z$; set
$\eta=|a_0+a_1\xi_1+a_2\xi_2|$.

To begin with, we claim that $\eta\ne0$. Indeed, if $\eta=0$, then $a_0=-\alpha_0a_2$ and $a_1=-\alpha_1a_2$,
since $1,\xi_1,\xi_2$ span a $\Q$-vector space of dimension~2\,---\,there is exactly one $\Q$-linear relation among them, up to proportionality.
As both
$$
D_na_1 \quad\text{and}\quad A\frac{D_{2n}}{D_{n}}\,a_1=-A\alpha_1\cdot\frac{D_{2n}}{D_{n}}\,a_2
$$
are integral, we have $\delta_na_1\in\Z$, where $\delta_n=\gcd(D_{n},AD_{2n}/D_{n})=e^{o(n)}$ as $n\to\infty$.
If $a_1\ne0$, the latter asymptotics leads to the contradiction with
$\delta_n\geq |a_1|^{-1}\geq e^{(\tau+\ve)n}\geq e^{\ve n}$, since $\tau\geq 0$.
Therefore, $a_1=0$, implying $a_2=0$ because $\xi_2\not\in\Q$; finally, $0=\eta=|a_0|$, which is impossible as $(a_0,a_1,a_2)\ne\bold0$.
This completes the proof of the claim that $\eta\ne0$.

We write now $\eta=|\hat a_0+\hat a_1\xi_1|$ with $\hat a_0=a_0+\alpha_0a_2$ and $\hat a_1=a_1+\alpha_1a_2$.

If $\hat a_1\ne0$, we have $AD_{n}^2D_{2n}\hat a_0\in\Z$ and $AD_{2n}\hat a_1\in\Z$. Set
$\tilde a_0=-\operatorname{sign}(\hat a_1)AD_{n}^2D_{2n}\hat a_0\in\Z$
and $\tilde a_1=AD_{n}^2D_{2n}|\hat a_1|\in D_{n}^2\N$. By the assumption,
$\tilde a_1>0$ implying
$e^{2n+o(n)}\leq\tilde a_1\leq e^{(4-\tau-\ve)n+o(n)}\leq e^{(4-\tau)n}$.
Thus, $(\log\tilde a_1)/(4-\tau)\leq n$
which ensures that $\psi(\tilde a_1)\mid D_{n}^2\mid\tilde a_1$. Since $\tilde a_1\geq e^{2n+o(n)}$
and $n$ is sufficiently large in terms of $\mu>\mu_{\psi}(\xi_1)$, we deduce
$$
AD_{n}^2D_{2n}\eta = |\tilde a_0-\tilde a_1\xi_1|
> \frac{1}{\tilde a_1^{\mu-1}} > e^{-(\mu-1)(4-\tau)n},
$$
so that $\eta>e^{-(4+(\mu-1)(4-\tau)+\nu)n}$ for $n$ sufficiently large.

If $\hat a_1=0$, we get $\eta=|a_0'|$. Since $\eta\ne0$, this
implies $AD_{n}^2D_{2n}\eta\in\N^*$ and thus $\eta>e^{-4n+o(n)}$. Furthermore,
from $\gamma_{\psi}\in[0,1)$ we deduce that
$\mu_{\psi}(\xi_1)\geq 2-\gamma_{\psi}>1$, so that $(\mu_{\psi}(\xi_1)-1)(4-\tau)>0$. Thus,
we have $\eta>e^{-(4+(\mu-1)(4-\tau)+\nu)n}$ for $n$ sufficiently large in this case as well.

Therefore, in both cases $4+(\mu-1)(4-\tau)+\nu\in E_\tau(\xi_1,\xi_2)$ for all $\mu>\mu_{\psi}(\xi_1)$ and all $\nu>0$.
Taking the infimum of $E_\tau(\xi_1,\xi_2)$ we obtain the desired inequality for $i=1$, and also for $i=2$
in view of $\mu_{\psi}(\xi_1)=\mu_{\psi}(\xi_2)$.

Finally, $\mu_{\psi}(\xi)=2-\gamma_{\psi}=2-{2}/(4-\tau)$ for almost all $\xi\in\R$ with
respect to the Lebesgue measure, completing the proof.
\end{proof}


\subsection({Rational approximation to \003\266(3) only})%
{Rational approximation to $\zeta(3)$ only}


Combining Theorem~\ref{majoration s} with Proposition~\ref{majoration mu_psi}, we deduce the following result.

\begin{proposition}\label{majoration mu_psi zeta}
For $\psi(q)=D_{n}^3$ with $n=\lfloor(\log q)/(4-\tau_0)\rfloor$ and $\tau_0$ defined in \textup{\eqref{eq-C}},
we have the upper bound
$$
\mu_{\psi}(\zeta(3))\leq 1.92357696\dotsc.
$$
\end{proposition}

Let us conclude with a few remarks on this result.

As shown in~\cite{Indaga},
Ap\'ery's proof of the irrationality of $\zeta(3)$ leads to the estimate $\mu_{\psi'}(\zeta(3))\leq 2$,
where $\psi'(q)=D_n^3$ with $n=\lf(\log q)/(4\log(1+\sqrt{2}))\rf$.
Since $4\log(1+\sqrt(2))>4-\tau_0$, this implies $\mu_{\psi}(\zeta(3))\leq 2$ with the function $\psi$
in Proposition~\ref{majoration mu_psi zeta}. Therefore,
Proposition~\ref{majoration mu_psi zeta} is slightly sharper than what follows from Ap\'ery's construction.

Proposition~\ref{majoration mu_psi zeta} can be adapted to $\zeta(2)$; namely, we have $\mu_{\tilde{\psi}}(\zeta(2))\leq 1.92$, where
$\tilde{\psi}(q)=D_nD_{2n}$ with $n=\lfloor(\log q)/(4-\tau_0)\rfloor$ and $\tau_0$ as before.
However, this result follows directly from Ap\'ery's construction~\cite{Indaga}: Ap\'ery's proof yields
$\mu_{\tilde{\psi}'}(\zeta(2))\leq 2$, where $\tilde{\psi}'(q)=D_n^2$ with $n=\lf(\log q/(5(\log(1+\sqrt{5})-\log 2))\rf$.
Using elementary methods (see~\cite{these}), this upper bound can be shown to imply
$\mu_{\tilde{\psi}}(\zeta(2))\leq 3.103$, which is greater than the one from Proposition \ref{majoration mu_psi zeta}.

In the notation above, the upper bound of Proposition~\ref{majoration mu_psi zeta} and its analogue for $\zeta(2)$ imply
$\mu_{\tilde{\psi}'}(\zeta(2))\leq 15.54$ and $\mu_{\psi'}(\zeta(3))\leq 8.85$: these upper bounds are worse than the ones followed
from Ap\'ery's construction.

Proposition~\ref{majoration mu_psi zeta} means that $\zeta(3)$ does not belong to the set of $\xi\in\R\setminus\Q$ satisfying
$\mu_{\psi}(\xi)>1.92\hdots$. This set has Hausdorff dimension equal to $0.0681457\hdots$; this is
smaller than the one obtained after Corollary~5 in~\cite{Indaga}.

Finally, for the function $\psi\in\mathcal E$ in Proposition~\ref{majoration mu_psi zeta}
we have that $\mu_{\psi}(\xi)= 1.03\hdots$ for almost all $\xi\in\R$.
Therefore, Proposition \ref{majoration mu_psi zeta} is still quite far from being
optimal, since $\zeta(3)$ is presumably a `generic' real number.


\subsection{Generalization}

Clearly, our Definition \ref{def exponent} admits a straightforward generalization, in which the three numbers $\xi_0=1$, $\xi_1=\zeta(2)$ and $\xi_2=\zeta(3)$
are replaced by a collection of $m+1$ real numbers $\xi_0,\xi_1,\dots,\xi_m$, where $m\geq1$.

\begin{Def}\label{def general}
Let $(\delta_{i,n})_{n\in\N}$ for $i=0,\dots,m$ be $m+1$ sequences of non-negative integers
such that $\delta_{i,n}\mid\delta_{i,n+1}$ for each $i\in\{0,1,\dots,m\}$ and all $n\in\N$,
and $\delta_{i,n}=e^{\gamma_in+o(n)}$ as $n\to\infty$, where $\gamma_0,\gamma_1,\dots,\gamma_m$ are positive real numbers.
Consider the sequence $\Lambda$ of lattices $(\Lambda_n)_{n\in\N}$ in $\R^{m+1}$ given by
$$
\Lambda_n=\frac1{\delta_{0,n}}\Z\oplus\frac1{\delta_{1,n}}\Z\oplus\dots\oplus\frac1{\delta_{m,n}}\Z,
$$
so that $\Lambda_n\subset\Lambda_{n+1}$. For $\tau\in\R$ define the generalized diophantine exponent
$s_{\tau,\Lambda}(\xi_0,\xi_1,\dots,\xi_m)$ to be the infimum of the set $E_{\tau,\Lambda}(\xi_0,\xi_1,\dots,\xi_m)$ of all $s\in\R$ with the following property: if $\ve>0$ and
$n$ is sufficiently large in terms of~$\ve$, then
$$
(a_0,a_1,\dots,a_m)\in\Lambda_n, \quad 0<\max\{|a_0|,|a_1|,\dots,|a_m|\}<e^{-(\tau+\ve)n}
$$
implies $|a_0\xi_0+a_1\xi_1+\dots+a_m\xi_m|>e^{-sn}$.
\end{Def}

It is not hard to verify that analogues of Propositions \ref{propriete de base}, \ref{majoration mu_psi}, \ref{majoration s_tau si
lie} and Corollary \ref{majoration mu} can be adapted to the generalized diophantine exponent. However, we have to stress that
our particular case treated above does not exactly fall under Definition~\ref{def general}, since the divisibility
$$
\frac{D_{2n}}{D_n}\biggm|\frac{D_{2n+2}}{D_{n+1}}
$$
is violated for general~$n$. This issue can be fixed by introducing the factors $e_n$ of `neglectful' growth such that
$$
\frac{D_{2n}}{D_n}\biggm|e_n\frac{D_{2n+2}}{D_{n+1}},
$$
similarly to what we have done in the proof of Theorem \ref{th} in Section~\ref{z23}, and, of course, Definition~\ref{def general} can be redesigned
to cover these circumstances. We do not feel strong about discussing these generalized concepts of diophantine exponent here by a very simple reason:
things become more abstract and complicated and, at the same time, lack meaningful examples.

\medskip
\noindent
\textbf{Acknowledgments.}
The authors would like to thank St\'ephane Fischler for making this project possible.


\bibliographystyle{alpha}
\bibliography{szeta}

\begin{thebibliography}{BBBC07}

\bibitem[Ap{\'e}79]{Apery}
Roger Ap{\'e}ry.
\newblock Irrationalit\'e de $\zeta(2)$ et $\zeta(3)$.
\newblock {\em Ast\'erisque}, 61:11--13, 1979.

\bibitem[BBBC07]{BBBC}
D.~H. Bailey, D.~Borwein, J.~M. Borwein, and R.~E. Crandall.
\newblock Hypergeometric forms for {I}sing-class integrals.
\newblock {\em Experiment. Math.}, 16(3):257--276, 2007.

\bibitem[BR01]{BallRivoal}
Keith Ball and Tanguy Rivoal.
\newblock Irrationalit\'e d'une infinit\'e de valeurs de la fonction z\^eta aux
  entiers impairs.
\newblock {\em Invent. Math.}, 146(1):193--207, 2001.

\bibitem[Dau14]{these}
Simon Dauguet.
\newblock {\em G\'en\'eralisation du crit\`ere d'ind\'ependance lin\'eaire de
  {N}esterenko}.
\newblock PhD thesis, Universit\'e Paris-Sud, 2014.
\newblock \url{http://www.math.u-psud.fr/~dauguet/These/These.html}.

\bibitem[Fis09]{Indaga}
St\'ephane Fischler.
\newblock Restricted rational approximation and {A}p\'ery-type constructions.
\newblock {\em Indagationes Mathematicae}, 20(2):201--215, 2009.

\bibitem[Hat95]{HataZeta2}
Masayoshi Hata.
\newblock A note on {B}eukers' integral.
\newblock {\em J. Austral. Math. Soc. Ser. A}, 58(2):143--153, 1995.

\bibitem[Hat00]{HataZeta3}
Masayoshi Hata.
\newblock A new irrationality measure for {$\zeta(3)$}.
\newblock {\em Acta Arith.}, 92(1):47--57, 2000.

\bibitem[Riv02]{Rivoal}
Tanguy Rivoal.
\newblock Irrationalit\'e d'au moins un des neuf nombres
  {$\zeta(5),\zeta(7),\dots,\zeta(21)$}.
\newblock {\em Acta Arith.}, 103(2):157--167, 2002.

\bibitem[RV96]{RhinViolaZeta2}
Georges Rhin and Carlo Viola.
\newblock On a permutation group related to {$\zeta(2)$}.
\newblock {\em Acta Arith.}, 77(1):23--56, 1996.

\bibitem[Sla66]{Sl}
Lucy~Joan Slater.
\newblock {\em Generalized hypergeometric functions}.
\newblock Cambridge University Press, Cambridge, 1966.

\bibitem[Zud01]{Zud01}
Wadim Zudilin.
\newblock One of the numbers $\zeta(5)$, $\zeta(7)$, $\zeta(9)$, $\zeta(11)$ is
  irrational.
\newblock {\em Russian Math. Surveys}, 56(4):774--776, 2001.

\bibitem[Zud04]{Zu04}
Wadim Zudilin.
\newblock Arithmetic of linear forms involving odd zeta values.
\newblock {\em J. {T}h\'eorie des {N}ombres de {B}ordeaux}, 16(1):251--291,
  2004.

\bibitem[Zud07]{Zu07}
Wadim Zudilin.
\newblock Approximations to -, di- and tri- logarithms.
\newblock {\em J. {C}omput. {A}ppl. {M}ath.}, 202(2):450--459, 2007.

\bibitem[Zud11]{Zud11}
Wadim Zudilin.
\newblock Arithmetic hypergeometric series.
\newblock {\em Russian Math. Surveys}, 66(2):369--420, 2011.

\bibitem[Zud14]{Zu13}
Wadim Zudilin.
\newblock Two hypergeometric tales and a new irrationality measure of
  $\zeta(2)$.
\newblock {\em Ann. Math. Qu\'ebec}, 38, 2014.
\newblock arXiv:1310.1526.

\end{thebibliography}


\end{document}